\documentclass[letter, 10pt, conference]{ieeeconf}      % Use this line for a4
\IEEEoverridecommandlockouts                              % This command is only
\usepackage[utf8]{inputenc}
\usepackage{url}
\usepackage{graphics}
\usepackage{graphicx}
\usepackage{amsmath}
\usepackage{amssymb}
\usepackage{mathrsfs}
\usepackage[ruled,commentsnumbered,norelsize]{algorithm2e}
\usepackage{cite}
\usepackage[usenames,dvipsnames]{color}
\usepackage{enumerate}
\usepackage{comment}
\usepackage{subfig}
\usepackage{multirow}
\usepackage[english]{babel}
\usepackage{ifpdf}
\usepackage{color}
\usepackage{nicefrac}
\usepackage{breqn}
\usepackage{pgfplots}

\newcommand{\prox}{\mathop{\rm prox}\nolimits}

\newcommand{\sign}{\mathop{\rm sign}\nolimits}

\newcommand{\minimize}{\operatorname{minimize}}  
\renewcommand{\Re}{{\rm{I\!R}} }

\newcommand{\Rinf}{\overline{\rm I\!R} }

\DeclareMathOperator*{\argmin}{\arg\!\min}
\newcommand{\stt}{\rm subject\ to}

\newcommand{\ie}{\emph{i.e.},~}

\definecolor{red}{rgb}{1,0,0}
\definecolor{yellow}{rgb}{1,1,0}
\definecolor{blue}{rgb}{0,0,1}

\newtheorem{theorem}{Theorem}
\newtheorem{lemma}{Lemma}
\newtheorem{proposition}{Proposition}

\title{\LARGE \bf Douglas-Rachford Splitting: Complexity Estimates and Accelerated Variants}

\author{Panagiotis Patrinos and Lorenzo Stella and Alberto Bemporad%
\thanks{The authors are with IMT Institute for Advanced Studies Lucca,
Piazza S. Ponziano 6, 55100 Lucca, Italy. Email: {\tt\small \{panagiotis.patrinos, lorenzo.stella, alberto.bemporad\}@imtlucca.it}}
\thanks{This work was partially supported by the European Research Project FP7-EFFINET (Grant no. 318556).}
}

\begin{document}

\maketitle
\thispagestyle{empty}
\pagestyle{empty}

%%%%%%%%%%%%%%%%%%%%%%%%%%%%%%%%%%%%%%%%%%%%%%%%%%%%%%%%%%%%%%%%%%%%%%%%%%%%%%%%
\begin{abstract}
We propose a new approach for analyzing convergence of the Douglas-Rachford
splitting method for solving convex composite optimization problems. The approach is
based on a continuously differentiable function, the \emph{Douglas-Rachford
Envelope (DRE)}, whose stationary points correspond to the solutions of the
original (possibly nonsmooth) problem. By proving the equivalence between the
Douglas-Rachford splitting method and a scaled gradient method applied to the DRE, results from
smooth unconstrained optimization are employed to analyze convergence properties
of DRS, to tune the method and to derive an accelerated version of it.
\end{abstract}

%%%%%%%%%%%%%%%%%%%%%%%%%%%%%%%%%%%%%%%%%%%%%%%%%%%%%%%%%%%%%%%%%%%%%%%%%%%%%%%%
\section{Introduction}\label{sec:Introduction}
In this paper we consider convex optimization problems of the form
\begin{equation}
{\minimize}\ F(x) = f(x)+g(x),
\label{eq:CompProb}
\end{equation}
where $f:\Re^n\to\Rinf$ and $g:\Re^n\to\Rinf$ are proper closed convex functions
with easily computable \emph{proximal mappings}~\cite{rockafellar1976monotone}. 
We recall that for a convex function $h:\Re^n\to\Rinf$ 
and positive scalar $\gamma$, the proximal mapping is defined as
\begin{equation}
\prox_{\gamma h}(x) = \argmin_z\left\{h(z)+\tfrac{1}{2\gamma}\|z-x\|^2\right\}.
\label{eq:ProxMap}
\end{equation}

A well known algorithm  for solving~\eqref{eq:CompProb} is 
the Douglas-Rachford splitting (DRS)  method~\cite{lions1979splitting}. 
In fact,  DRS can be applied to solve the more general problem of finding 
the zero of two maximal monotone operators. In the special case where the 
corresponding operators are the subdifferentials of $f$ and $g$, DRS amounts
to the following iterations 
\begin{subequations}\label{eq:DRS}
\begin{align}
y^k &= \prox_{\gamma f}(x^k),\label{eq:DRSfirst}\\
z^k &= \prox_{\gamma g}(2y^k-x^k),\\
x^{k+1} &= x^k + \lambda_k(z^k-y^k),\label{eq:DRSlast}
\end{align}  
\end{subequations}
where $\gamma > 0$ and the stepsizes $\lambda_k\in[0,2]$ satisfy
$\sum_{k\in\mathbb{N}}\lambda_k(2-\lambda_k)=+\infty$. A typical choice for $\lambda_k$ is to be set equal to $1$ for
all $k$. If the minimum in~\eqref{eq:CompProb} is attained and
the relative interiors of the effective domains of $f$ and $g$ have a point in common,
then it is well known that
$\{z^k-y^k\}$ converges to $0$, and $\{x^k\}$ converges to $x$ such that
$\prox_{\gamma f}(x)\in\argmin F$~\cite{eckstein1989splitting, eckstein1992douglasrachford,bauschke2011convex}. 
Therefore $\{y^k\}$ and $\{z^k\}$ converge to a
solution of~\eqref{eq:CompProb}. This general form of DRS was proposed by~\cite{eckstein1989splitting, eckstein1992douglasrachford}, where it was
shown that DRS is a particular case of the proximal point algorithm~\cite{rockafellar1976monotone}. 
Thus DRS converges under very general assumptions. For example, unlike forward-backward splitting (FBS)~\cite{combettes2011proximal},
it does not require differentiability of one of the two summands and parameter $\gamma$ can take any positive value.

Another well-known application of DRS is for solving problems of the form
\begin{align}\label{eq:CompProbCons}
{\minimize}&\ f(x)+g(z),\\
\stt&\ Ax+Bz=b.\nonumber
\end{align}
Applying DRS to the dual of problem~\eqref{eq:CompProbCons} leads to the
alternating direction method of multipliers (ADMM)~\cite{gabay1983applications,eckstein1989splitting, eckstein1992douglasrachford}. This method has
recently received a lot of attention, especially because of its properties with
respect to separable objective functions, that make it favorable for large-scale
problems and distributed applications~\cite{boyd2011distributed,parikh2013proximal}.
 
However, when applied to~\eqref{eq:CompProb}, the behavior of DRS is quite different compared to standard
optimization methods. For example, unlike FBS, DRS is not a descent method, in that the sequence of cost values $\{F(x^k)\}$ may not be monotone decreasing. This is perhaps one of the main reasons why the convergence rate of DRS has not been well understood and convergence rate results were scarce, until very recently. The first convergence result for DRS appeared
in~\cite{lions1979splitting}. Translated to the setting of solving~\eqref{eq:CompProb}, under strong convexity and Lipschitz continuity
assumptions for $f$, the sequence $\{x^k\}$ was shown to converge $Q$-linearly to the (unique)
optimal solution of~\eqref{eq:CompProb}. More recently, it was shown that if $f$ is differentiable then the squared residual $\|x^k-\prox_{\gamma g}(x^k-\gamma\nabla f(x^k))\|^2$ converges to zero with sublinear rate of $1/k$~\cite{he2012o1n}. In~\cite{goldstein2012fast} convergence rates of order $1/k$ for the objective values  are provided implicitly for DRS under the assumption that both $f$ and $g$ have Lipschitz continuous gradients. Under the additional assumption that $f$ is quadratic, the authors of~\cite{goldstein2012fast} give an accelerated version with convergence rate $1/k^2$.  In~\cite{deng2012global} the authors show global linear convergence for ADMM under a variety of scenarios. Translated in the DRS setting, they require at least $f$ to be strongly convex with Lipschitz continuous gradient. In~\cite{hong2012linear}  $R$-linear convergence of the duality gap and primal cost for multiple splitting ADMM under less stringent assumptions is shown, provided that the stepsizes $\lambda_k$ are sufficiently small. However, the form of the convergence rate is not very informative, since the bound on the stepsizes depends on constants that are very hard to compute. In~\cite{ghadimi2013optimal} it is shown that ADMM converges linearly for quadratic programs with the constraint matrix being full rank. However explicit complexity estimates are only provided for the (infrequent) case where the constraint matrix is full row rank. Convergence rates of DRS and ADMM are analyzed under various assumptions in the recent paper \cite{davis2014convergence}.

%\note{Mention existing convergence results and acceleration techniques}
\subsection{Our contribution}
In this paper we follow a new approach to the analysis of the convergence
properties and complexity estimates of DRS.
We show that when $f$ is twice continuously differentiable, then problem~\eqref{eq:CompProb}
is equivalent to computing a stationary point of a continuously differentiable function, the
\emph{Douglas-Rachford Envelope (DRE)}. Specifically, DRS is shown to be nothing more than a (scaled)
gradient method applied to the DRE. This kind of interpretation is similar to the one offered by the
Moreau envelope for the proximal point algorithm and paves the way for deriving new algorithms based
on the Douglas-Rachford splitting approach.
%\red{In this sense the DRE is to the
%Douglas-Rachford splitting as the Moreau envelope is to the proximal point algorithm.}

A similar idea has been exploited
in~\cite{patrinos2013proximal,patrinos2014forwardbackward}
in order to express another splitting method, the forward-backward splitting,
as a gradient method applied to the so-called Forward-Backward Envelope (FBE). 
There the purpose was use the FBE as a merit function on which to perform
Newton-like methods with superlinear local convergence rates to solve non
differentiable problems. Here the purpose is instead to analyze the convergence rate
properties of Douglas-Rachford splitting by expressing it as a gradient
method. Specifically, we show that if $f$ is convex quadratic (but $g$ can still
be any convex nonsmooth function) then the DRE
is convex with Lipschitz continuous gradient, provided that $\gamma$ is sufficiently small. 
This covers a wide variety of problems such as quadratic programs, $\ell_1$ least squares, 
nuclear norm regularized least squares, image restoration/denoising problems involving total variation
minimization norm, etc.
This observation makes convergence rate analysis of DRS extremely easy, since it allows us 
to directly apply the well known complexity estimates of the gradient method. 
Furthermore, we discuss the optimal choice of the parameter
$\gamma$ and of the stepsize $\lambda_k$ defining the method, and devise a method with faster convergence
rates %by performing a line search to determine the stepsize and
by exploiting
the acceleration techniques introduced by Nesterov~\cite{nesterov1983method},\cite[Sec. 2.2]{nesterov2003introductory}.

%\note{Sketch the idea, mention FBE}

%\note{More specific assumptions that we make}

%\note{DR splitting as scaled unconstrained gradient method}

%\note{Fast gradient method on the DRE}

%\note{Structure of the paper}

The paper is structured as follows.
In Section~\ref{sec:DRE} we define the Douglas-Rachford envelope and analyze its
properties, illustrating how DRS is equivalent to a scaled gradient
method applied to the DRE. Section~\ref{sec:Quadratic} discusses the convergence
of Douglas-Rachford splitting in the particular but important case in which
$f$ is convex quadratic, where the DRE turns out to be convex.
Section~\ref{sec:Algorithms} considers the application of accelerated gradient
methods to the DRE to achieve faster convergence rates. Finally,
Section~\ref{sec:Simulations} shows experimental results obtained with the proposed methods.

%%%%%%%%%%%%%%%%%%%%%%%%%%%%%%%%%%%%%%%%%%%%%%%%%%%%%%%%%%%%%%%%%%%%%%%%%%%%%%%%
\section{Douglas-Rachford Envelope}\label{sec:DRE}
 We indicate by $X_\star$ the set of optimal solutions to problem~\eqref{eq:CompProb}, which we assume to be nonempty. Then $x_\star\in X_\star$ if and only if~\cite[Cor. 26.3]{bauschke2011convex} $x_\star=\prox_{\gamma f}(\tilde x)$, where $\tilde{x}$ is a solution of 
\begin{align}
\prox_{\gamma g}(2\prox_{\gamma f}(x)-x)-\prox_{\gamma f}(x)=0.\label{eq:OptCond2}
\end{align}
Let $\tilde X$ be the set of solutions to~\eqref{eq:OptCond2}.
Our goal is to find a continuously differentiable function whose set of
stationary points is equal to $\tilde X$.

Given a function $h:\Re^n\to\Rinf$, consider its \emph{Moreau envelope}
$$h^\gamma(x)=\inf_z\left\{h(z)+\tfrac{1}{2\gamma}\|z-x\|^2\right\}.$$
It is well known that $h^\gamma:\Re^n\to\Re$ is differentiable (even if $h$ is nonsmooth) with $(1/\gamma)$-Lipschitz continuous
gradient
\begin{equation}\label{eq:MoreauGrad}
\nabla h^\gamma(x)= \gamma^{-1}(x-\prox_{\gamma h}(x)).
\end{equation}
By using~\eqref{eq:MoreauGrad} we can rewrite~\eqref{eq:OptCond2} as
\begin{equation}\label{eq:OptCond3}
\nabla f^{\gamma}(x)+\nabla g^\gamma(x-2\gamma\nabla f^{\gamma}(x))=0.
\end{equation}
From now on we make the extra assumption that $f$ is twice continuously differentiable, with
$L_f$-Lipschitz continuous gradient. We also assume that $f$ has strong convexity modulus equal to $\mu_f\geq 0$, i.e., function $f(x)-\tfrac{\mu_f}{2}\|x\|^2$ is convex.
Notice that we allow $\mu_f$ to be equal to zero, including also the case where $f$ is not strongly convex.
Due to these assumptions we have
\begin{equation}\label{eq:bndHess}
\|\nabla^2 f(x)\| \leq L_f,\ \mbox{for all }x\in\Re^n.
\end{equation}
Moreover, from \cite[Prop. 4.1, Th. 4.7]{lemarechal1997practical} the Jacobian
of $\prox_{\gamma f}$ and the Hessian of $f^\gamma$ exist everywhere and are related
to each other as follows:
\begin{align}
\nabla\prox_{\gamma f}(x)&=(I+\gamma\nabla^2f(\prox_{\gamma f}(x)))^{-1},\\
\nabla^2 f^\gamma(x)&=\gamma^{-1}(I-\nabla\prox_{\gamma f}(x)).\label{eq:MoreauHess}
\end{align}
Using \eqref{eq:bndHess}-\eqref{eq:MoreauHess} one can easily show that for any $d\in\Re^n$
\begin{equation}\label{eq:MoreauEig}
\tfrac{\mu_f}{1+\gamma\mu_f}\|d\|^2\leq d'\nabla^2 f^\gamma(x)d\leq\tfrac{L_f}{1+\gamma L_f}\|d\|^2.
\end{equation}
In other words, if $f$ is twice continuously differentiable with $L_f$-Lipschitz continuous gradient then the eigenvalues of the Hessian of its Moreau envelope are bounded uniformly for every $x\in\Re^n$.

Next, we premultiply~\eqref{eq:OptCond3} by $(I-2\gamma\nabla^2f^\gamma(x))$ to obtain the gradient
of what we call the \emph{Douglas-Rachford Envelope (DRE)}:
%$$ \nabla F^{\mathrm{DR}}_\gamma(x) = (I-2\gamma\nabla^2f^\gamma(x))(\nabla f^{\gamma}(x)+\nabla g^\gamma(x-2\gamma\nabla f^{\gamma}(x))),$$
\begin{equation}\label{eq:DRE1}
F^{\mathrm{DR}}_\gamma(x) = f^{\gamma}(x)-\gamma\|\nabla f^\gamma(x)\|^2+g^\gamma(x-2\gamma\nabla f^{\gamma}(x)).
\end{equation}
If $(I-2\gamma\nabla^2f^\gamma(x))$ is nonsingular for every $x$,
then every stationary point of $F^{\mathrm{DR}}_\gamma$ is also an element of
$\tilde X$, and vice versa. From~\eqref{eq:MoreauEig} we obtain
\begin{equation}\label{eq:EigH}
\tfrac{1-\gamma L_f}{1+\gamma L_f}\|d\|^2\leq d'(I-2\gamma\nabla^2f^\gamma(x))d\leq\tfrac{1-\gamma \mu_f}{1+\gamma \mu_f}\|d\|^2.
\end{equation}
Therefore whenever $\gamma<1/L_f$ or $\gamma>1/\mu_f$
(in case where $\mu_f>0$), finding a stationary point of the DRE~\eqref{eq:DRE1}
is equivalent to solving~\eqref{eq:OptCond2}.

It is convenient now to introduce the following notation:
\begin{align*}
P_\gamma(x)&=\prox_{\gamma f}(x),\\
G_\gamma(x)&=\prox_{\gamma g}(2P_\gamma(x)-x),\\
Z_\gamma(x)&=P_\gamma(x)-G_\gamma(x),
\end{align*}
so that condition~\eqref{eq:OptCond2} is expressed as $Z_\gamma(x)=0$.
By~\eqref{eq:MoreauHess} we can rewrite $ I-2\gamma\nabla^2f^\gamma(x) = 2\nabla P_\gamma(x) - I$, therefore 
the gradient of the DRE can be expressed as
\begin{equation}\label{eq:GradDRE}
\nabla F^{\mathrm{DR}}_\gamma(x) = \gamma^{-1}(2\nabla P_\gamma(x)-I)Z_\gamma(x).
\end{equation}

%\note{Briefly discuss relation between FBE \& DRE}
%\note{General properties of the DRE}

The following proposition is instrumental in establishing an equivalence between problem~\eqref{eq:CompProb} and
that of minimizing the DRE.
\begin{proposition}\label{prop:BasicInequalities}
The following inequalities %involving $P_\gamma$, $G_\gamma$ and $F^{\mathrm{DR}}_\gamma$
hold for any $\gamma>0$ and $x\in\Re^n$:
\begin{subequations}
\begin{align}
F^{\mathrm{DR}}_\gamma(x)&\leq F(P_\gamma(x))-\tfrac{1}{2\gamma}\|Z_\gamma(x)\|^2,\label{eq:Ineq1}\\
F^{\mathrm{DR}}_\gamma(x)&\geq F(G_\gamma(x))+\tfrac{1-\gamma L_f}{2\gamma}\|Z_\gamma(x)\|^2.\label{eq:Ineq2}
\end{align}
\end{subequations}
\end{proposition}
\begin{proof}
%See \cite{patrinos2014douglas}.
See Appendix.
\end{proof}

The following fundamental result shows, under the assumption
of $\gamma$ being sufficiently small, that minimizing the DRE, which is real-valued and smooth,
is completely equivalent to solving the nonsmooth problem~\eqref{eq:CompProb}.
Furthermore, the set of stationary points of the DRE, which may not be convex,
coincide with the set of its minimizers.
\begin{theorem}\label{th:Equiv}
If $\gamma\in(0,1/L_f)$ then
\begin{align*}
\inf\ F    &= \inf\ F^{\mathrm{DR}}_\gamma,\\
\argmin\ F &= P_\gamma(\argmin\ F^{\mathrm{DR}}_\gamma).
\end{align*}
\end{theorem}
\begin{proof}
By~\cite[Cor. 26.3]{bauschke2011convex} we know that $x_\star\in X_\star$ if and
only if $x_\star = P_\gamma(\tilde x)$, for some $\tilde x\in\tilde{X}$, \ie
with $P_\gamma(\tilde x)=G_\gamma(\tilde x)$.
Putting $x=\tilde{x}$ in~\eqref{eq:Ineq1},~\eqref{eq:Ineq2} one obtains
$$ F^{\mathrm{DR}}_\gamma(\tilde x) = F(x_\star).$$
When $\gamma < 1/L_f$, Eq.~\eqref{eq:Ineq2} implies that for all $x\in\Re^n$
\begin{equation}\label{eq:IneqFgamma}
F^{\mathrm{DR}}_\gamma(x) \geq F(G_\gamma(x)) \geq F(x_\star)=F^{\mathrm{DR}}_\gamma(\tilde x),
\end{equation}
where the last inequality follows from optimality of $x_\star$. Therefore
the elements of $\tilde{X}$ are minimizers of $F^{\mathrm{DR}}_\gamma$ and
$\inf F = \inf F^{\mathrm{DR}}_\gamma$.
They are indeed the only minimizers, for if $x\notin\tilde{X}$ 
then $Z_\gamma(x) \neq 0$ in \eqref{eq:Ineq2}, and the first inequality in~\eqref{eq:IneqFgamma} is strict.
\end{proof}
\subsection{DRS as a variable-metric gradient method}
In simple words, Theorem~\ref{th:Equiv} tells us that under suitable assumptions on $\gamma$, one can employ whichever
smooth unconstrained optimization technique for minimizing the DRE and thus solve~\eqref{eq:CompProb}.
The resulting algorithm will of course bear a close relationship to DRS since the gradient of the DRE, cf.~\eqref{eq:GradDRE}, is inherently related to a step of DRS, cf.~\eqref{eq:DRS}. 

In particular, from the expression~\eqref{eq:GradDRE} for $\nabla F^{\mathrm{DR}}_\gamma$,
one observes that Douglas-Rachford splitting can be interpreted as a variable-metric
gradient method for minimizing $F^{\mathrm{DR}}_\gamma$. Specifically, we have
that the $x$-iterates defined by~\eqref{eq:DRS} correspond to
\begin{equation}
x^{k+1}=x^k-\lambda_kD^k\nabla F^{\mathrm{DR}}_\gamma(x^k),\label{eq:ScaledGrad}
\end{equation}
where
\begin{equation}
D^k=\gamma(2\nabla P_\gamma(x^k)-I)^{-1}.\label{eq:ScalingMat}
\end{equation}
We can then exploit all the well known convergence results of gradient
methods to analyze the properties of DRS or propose alternative schemes of it. 
\subsection{Connection between DRS and FBS}
The DRE reveals an interesting link between Douglas-Rachford splitting and forward-backward splitting, that has remained unnoticed at least to our knowledge.
Let us  first derive an alternative way of expressing the DRE. Since $P_\gamma(x)=\argmin_z\{f(z)+\tfrac{1}{2}\|z-x\|^2\}$ satisfies
\begin{equation}\label{eq:optP}
\nabla f(P_\gamma(x))+\gamma^{-1}(P_\gamma(x)-x)=0,
\end{equation}
the gradient of the Moreau envelope of $f$ becomes
\begin{align}
\nabla f^{\gamma}(x)= \gamma^{-1}(x-P_\gamma(x))=\nabla f(P_\gamma(x)).\label{eq:MoreauGrad2}
\end{align}
Using~\eqref{eq:optP},~\eqref{eq:MoreauGrad2} in~\eqref{eq:DRE1} we obtain the following alternative expression for the DRE
%f^\gamma(x)%&=f(P_\gamma(x))+\tfrac{1}{2\gamma}\|P_\gamma(x)-x\|^2 \\
%    &=f(P_\gamma(x))+\tfrac{\gamma}{2}\|\nabla f(P_\gamma(x))\|^2, \\
\begin{equation}
F^{\mathrm{DR}}_\gamma{=}f(P_\gamma(x))-\tfrac{\gamma}{2}\|\nabla f(P_\gamma(x))\|^2+g^\gamma(2P_\gamma(x)-x),\label{eq:DRE2}
\end{equation}
Next, using the definition of $g^\gamma$ in~\eqref{eq:DRE2}, it is possible to express
\begin{align}
F^{\mathrm{DR}}_\gamma(x) &= \min_{z\in\Re^n}\{f(P_\gamma(x))+\nabla f(P_\gamma(x))'(z-P_\gamma(x))\nonumber\\
    &\phantom{= \min_{z\in\Re^n}\{}+g(z)+\tfrac{1}{2\gamma}\|z-P_\gamma(x)\|^2\}.\label{eq:DRE3}
\end{align}
Comparing this with the definition of the forward-backward envelope (FBE)
introduced in~\cite{patrinos2013proximal}
$$F^{\mathrm{FB}}_\gamma(x)= \min_{z\in\Re^n}\{f(x)+\nabla f(x)'(z-x)+g(z)+\tfrac{1}{2\gamma}\|z-x\|^2\},$$
it is apparent that the DRE at $x$ is equal to the FBE evaluated at $P_\gamma(x)$:
$$ F^{\mathrm{DR}}_\gamma(x) = F^{\mathrm{FB}}_\gamma(P_\gamma(x)).$$
Let us recall here that iterates $x^{k+1}$ of FBS are obtained by solving the optimization problem appearing in the definition
of FBE for $x=x^k$.
Therefore, it can be easily seen that an iteration of DRS corresponds to a forward-backward step applied to $\prox_{\gamma f}(x^k)$ (instead of $x^k$, as in FBS).
%\note{The quadratic case}

\section{Douglas-Rachford Splitting}\label{sec:Quadratic}
In case $f$ is convex quadratic, \ie
$$ f(x) = \tfrac{1}{2}x'Qx + q'x, $$
with $Q\in\Re^{n\times n}$ symmetric and positive semidefinite
and $q\in\Re^n$, we have
\begin{align}
P_\gamma(x) &= (I+\gamma Q)^{-1}(x-\gamma q),\label{eq:Pquad}\\
\nabla P_\gamma(x)&= (I+\gamma Q)^{-1}.\label{eq:gradPquad}
\end{align}

%\note{More things that are worth noticing in this case}

We now have $\mu_f = \lambda_{\min}(Q)$ and $L_f = \lambda_{\max}(Q)$.
It turns out that in this case, under the
already mentioned assumption $\gamma<1/L_f$, the DRE
is convex.
\begin{theorem}\label{th:QuadConvex}
Suppose that $f$ is convex quadratic.
If $\gamma < 1/L_f$, then $F^{\mathrm{DR}}_\gamma$ is convex with $L_{F^{\mathrm{DR}}_\gamma}$-Lipschitz continuous gradient and convexity modulus $\mu_{F^{\mathrm{DR}}_\gamma}$ given by
\begin{align}
L_{F^{\mathrm{DR}}_\gamma}&=\frac{1-\gamma \mu_f}{1+\gamma \mu_f}\gamma^{-1}\label{eq:LDRE},\\
\mu_{F^{\mathrm{DR}}_\gamma}&=\min\left\{\frac{(1-\gamma\mu_f)\mu_f}{(1+\gamma\mu_f)^2},\frac{(1-\gamma L_f)L_f}{(1+\gamma L_f)^2}\right\}.\label{eq:mDRE}
\end{align}
\end{theorem}
\begin{proof}
%See \cite{patrinos2014douglas}.
%\begin{comment}
Using~\eqref{eq:GradDRE},~\eqref{eq:gradPquad},~\eqref{eq:EigH}  and Lemma~\ref{le:NonexpZ} in the Appendix, we obtain
\begin{align*}
\|\nabla F^{\mathrm{DR}}_\gamma(x_1)-\nabla F^{\mathrm{DR}}_\gamma(x_2)\|&\leq \gamma^{-1}\|2(I+\gamma Q)^{-1}-I\|\\
&\phantom{\leq \gamma^{-1}}\cdot\|Z_\gamma(x_1)-Z_\gamma(x_2)\|\\
&\leq\left(\tfrac{2}{1+\gamma \mu_f}-1\right)\gamma^{-1}\|x_1-x_2\|.
\end{align*}

Next, due to the form of $P_\gamma$, cf.~\eqref{eq:Pquad} it is evident that
$f(P_\gamma(x))-\frac{\gamma}{2}\|\nabla f(P_\gamma(x))\|^2$ is
quadratic with Hessian
$$H=(I+\gamma Q)^{-1}(I-\gamma Q)Q (I+\gamma Q)^{-1}.$$
The eigenvalues of $H$ are given by $\tfrac{(1-\gamma\lambda_i)\lambda_i}{(1+\gamma\lambda_i)^2}$, where $\lambda_i$, $i=1,\ldots,n$ are the eigenvalues of $Q$.
%$$\sigma(H) = \left\{\tfrac{(1-\gamma\lambda_i)\lambda_i}{(1+\gamma\lambda_i)^2}\left|\right. \lambda_i\in\sigma(Q),\ i=1,\ldots,n\right\}.$$
%where $\lambda_i$ are the eigenvalues of $Q$.
Consider the function
$$\psi(\lambda)=\tfrac{(1-\gamma\lambda)\lambda}{(1+\gamma\lambda)^2}.$$
If $\gamma < 1/L_f$, $\psi$ is concave and its minimum is
attained in one of the two endpoints of the interval $[\mu_f,L_f]$. 
The minimum eigenvalue of $f(P_\gamma(x))-\frac{\gamma}{2}\|\nabla f(P_\gamma(x))\|^2$
is then given by~\eqref{eq:mDRE}.
On the other hand, $g^\gamma(x-2\gamma\nabla f^{\gamma}(x))$ is convex as the
composition of the convex function $g^\gamma$ with an affine map. Therefore, the DRE as expressed by~\eqref{eq:DRE2},
is the sum of two functions, one of them being (strongly) convex with modulus $\mu_{F^{\mathrm{DR}}_\gamma}$ and the other convex. Hence it is (strongly) convex with modulus $\mu_{F^{\mathrm{DR}}_\gamma}$.
%\end{comment}
\end{proof}

Therefore, under the assumptions of Theorem~\ref{th:QuadConvex}, we can exploit the
well known results on the convergence of the gradient method for convex problems.
To do so, note that when $f$ is quadratic, $P_\gamma$ is linear and
the scaling matrix $D^k$ defined in~\eqref{eq:ScalingMat} is constant, \ie
$$D^k \equiv D = \gamma(2(I+\gamma Q)^{-1}-I)^{-1}.$$
Consider the linear change of variables $x = Sw$, where $S = D^{1/2}$. Note that
\begin{equation}\label{eq:eigenD}
\lambda_{\mathrm{min}}(D) = \gamma\frac{1+\gamma \mu_f}{1-\gamma \mu_f},
\quad\lambda_{\mathrm{max}}(D) = \gamma\frac{1+\gamma L_f}{1-\gamma L_f},
\end{equation}
so if $\gamma < 1/L_f \leq 1/\mu_f$ then
matrix $D$ is positive definite and $S$ is well defined.

In the new variable $w$, the scaled gradient iterations~\eqref{eq:ScaledGrad}
correspond to the (unscaled) gradient method applied to the preconditioned
problem
$$ \minimize\ h(w) = F^{\mathrm{DR}}_\gamma(Sw). $$
Indeed, the gradient method applied on $h$ is
\begin{equation}
w^{k+1} = w^k-\lambda_k\nabla h(w^k)\label{eq:GMh}
\end{equation}
Multiplying by $S$ and using $\nabla h(w^k)=S\nabla F^{\mathrm{DR}}_\gamma(Sw^k)$, 
we obtain
$$ x^{k+1} = x^k-\lambda_k D\nabla F^{\mathrm{DR}}_\gamma(x^k).$$
Recalling~\eqref{eq:GradDRE}, this becomes
$$ x^{k+1} = x^k-\lambda_k Z_\gamma(x^k),$$
which is exactly DRS, cf.~\eqref{eq:DRS}.
From now on we will indicate by $\tilde w$ a minimizer of $h$, so that
$\tilde w = S\tilde x$ for some $\tilde x\in\tilde X$.
From Theorem~\ref{th:QuadConvex} we know that if $\gamma < 1/L_f$ then $F_{\gamma}^{\mathrm{DR}}$
is convex with Lipschitz continuous gradient, and so is $h$. In particular,
\begin{align}
\mu_h &= \lambda_{\min}(D)\mu_{F_{\gamma}^{\mathrm{DR}}},\label{eq:mh}\\
L_h &= \lambda_{\max}(D)L_{F_{\gamma}^{\mathrm{DR}}} = \frac{1+\gamma L_f}{1-\gamma L_f}.
\end{align}

\begin{theorem}\label{th:QuadConvergence1}
For convex quadratic $f$, if $\gamma < 1/L_f$ and
\begin{equation}\label{eq:lambdass}
\lambda_k = \lambda = (1-\gamma L_f)/(1+\gamma L_f)
\end{equation}
then the sequence of iterates generated by~\eqref{eq:DRSfirst}-\eqref{eq:DRSlast} satisfies
$$ F(z^{k+1})-F_\star \leq \frac{1}{(2\gamma\lambda) k}\|x^0-\tilde x\|^2. $$
\end{theorem}
\begin{proof}
Douglas-Rachford splitting~\eqref{eq:DRS} corresponds to the gradient descent
iterations \eqref{eq:GMh}. %,
%via the change of variable $w = S^{-1}x$.
So by setting $\lambda = 1/L_h$ one has:
$$ h(w^{k})-h(\tilde{w})\leq\frac{L_h}{2 k}\|w^0-\tilde w\|^2, $$
see for example \cite[Prop. 6.10.2]{bertsekas2009online}.
Applying the substitution $x = Sw$, and considering that
\begin{equation}\label{eq:boundsD}
\lambda_{\max}^{-1}(D)\|x\|^2 \leq \|x\|_{D^{-1}}^2 \leq \lambda_{\min}^{-1}(D)\|x\|^2,\ \forall x\in\Re^n
\end{equation}
one obtains
\begin{align*}
F_{\gamma}^{\mathrm{DR}}(x^k)-F_{\gamma}^{\mathrm{DR}}(\tilde x)&\leq\frac{L_h}{2 k}\|x^0-\tilde x\|_{D^{-1}}^2\\
&\leq\frac{1}{2 k}\frac{1+\gamma L_f}{(1-\gamma L_f)}\frac{1}{\lambda_{\min}(D)}\|x^0-\tilde x\|^2\\
&=\frac{1}{2 k}\frac{1+\gamma L_f}{\gamma(1-\gamma L_f)}\|x^0-\tilde x\|^2,
\end{align*}
where the last equality holds considering \eqref{eq:eigenD}.
The claim follows by $z^k = G_\gamma(x^k)$, Theorem~\ref{th:Equiv} and inequality~\eqref{eq:Ineq2}.
\end{proof}

From Theorem~\ref{th:QuadConvergence1} we easily obtain the following optimal
value of $\gamma$:
\begin{equation}\label{eq:OptimalGamma}
\gamma_\star = \argmin_\gamma\ \frac{1+\gamma L_f}{\gamma(1-\gamma L_f)} = \frac{\sqrt{2}-1}{L_f}.
\end{equation}
For this particular value of $\gamma_\star$ the stepsize becomes equal to $\lambda_k=\sqrt{2}-1$.
In the strongly convex case we instead obtain the following stronger result.

\begin{theorem}\label{th:QuadConvergence2}
If $\mu_f>0$ and $\lambda_k=\lambda\in (0,2/(L_{h}+\mu_{h})]$ then
$$\|y^{k}-x_\star\|^2\leq\frac{\lambda_{\max}(D)}{\lambda_{\min}(D)}\left(1-\frac{2\lambda\mu_h L_h}{\mu_h+L_h}\right)^k\|x^0-\tilde x\|^2.$$
%In particular if $\lambda_k=1/(\mu_{h}+L_{h})$ then 
%$$\|y^{k+1}-x_\star\|\leq\frac{\lambda_{\max}(D)}{\lambda_{\min}(D)}\left(\frac{\kappa_{h}-1}{\kappa_{h}+1}\right)^k\|x^0-x_\star\|$$
%where $\kappa_{h}=L_{h}/\mu_{h}$.
\end{theorem}
\begin{proof}
Just like in the proof of Theorem~\ref{th:QuadConvergence1},
iteration~\eqref{eq:GMh} is the standard gradient method applied to $h$. If $f$ is strongly
convex then we have, using \eqref{eq:mDRE} and \eqref{eq:mh}, that also $h$ is strongly convex.
From~\cite[Th. 2.1.15]{nesterov2003introductory} we have 
$$\|w^{k}-\tilde w\|^2\leq\left(1-\frac{2\lambda\mu_h L_h}{\mu_h+L_h}\right)^k\|w^0-\tilde w\|^2.$$
Applying the substitution $x = Sw$ we get
$$\|x^{k}-\tilde x\|_{D^{-1}}^2\leq\left(1-\frac{2\lambda\mu_h L_h}{\mu_h+L_h}\right)^k\|x^0-\tilde x\|_{D^{-1}}^2.$$
The thesis follows considering \eqref{eq:boundsD} and that
$$\|y^{k}-x_\star\|^2 = \|\prox_{\gamma f}(x^{k})-\prox_{\gamma f}(\tilde x)\|^2 \leq \|x^{k}-\tilde x\|^2,$$
where the equality holds since $x_\star=\prox_{\gamma f}(\tilde{x})$, and the inequality by nonexpansiveness
of $\prox_{\gamma f}$.
\end{proof}

%%%%%%%%%%%%%%%%%%%%%%%%%%%%%%%%%%%%%%%%%%%%%%%%%%%%%%%%%%%%%%%%%%%%%%%%%%%%%%%%
\section{Fast Douglas-Rachford splitting}\label{sec:Algorithms}

We have shown that DRS is equivalent to the
gradient method minimizing $h(w) = F^{\mathrm{DR}}_\gamma(Sw)$. In
the quadratic case, since for $\gamma < 1/L_f$ we know that $F^{\mathrm{DR}}_\gamma(x)$
is convex, we can as well apply the optimal first order methods
due to Nesterov~\cite{nesterov1983method}, \cite[Sec. 2.2]{nesterov2003introductory}
to the same problem. This way we obtain a \emph{fast Douglas-Rachford splitting}
method. The scheme is as follows: given $u^0=x^0\in\Re^n$, iterate
\begin{subequations}
\begin{align}
y^k&=\prox_{\gamma f}(u^k),\label{eq:FDRSfirst}\\
z^k&=\prox_{\gamma g}(2y^k-u^k),\\
x^{k+1}&=u^k+\lambda_k(z^k-y^k),\\
u^{k+1}&=x^{k+1}+\beta_k(x^{k+1}-x^k).\label{eq:FDRSlast}
\end{align}
\end{subequations}
%where
%\begin{align*}
%\beta_k &= \frac{\alpha_k(1-\alpha_k)}{\alpha_k^2+\alpha_{k+1}},\\
%\alpha_{k+1}^2 &= (1-\alpha_{k+1})\alpha_k^2+\frac{\mu_h}{L_h}\alpha_{k+1}^2,\\
%\alpha_0 &\in (0,1),%&= \tfrac{\sqrt{5}-1}{2},
%\end{align*}
We have the following estimates regarding the convergence rate of
iterations~\eqref{eq:FDRSfirst}-\eqref{eq:FDRSlast}, whose proofs are
based on~\cite{nesterov2003introductory}.

\begin{theorem}
For convex quadratic $f$, if $\gamma < 1/L_f$, $\lambda_k$ are given by~\eqref{eq:lambdass} and
\begin{align*}
%\lambda_k &= \lambda=(1-\gamma L_f)/(1+\gamma L_f),\\
\beta_k &= \begin{cases} 0 & \mbox{if }k=0, \\ \tfrac{k-1}{k+2} & \mbox{if }k>0,\end{cases}
\end{align*}
then the sequence of iterates generated by~\eqref{eq:FDRSfirst}-\eqref{eq:FDRSlast} satisfies
\begin{equation*}
F(z^k)-F_\star \leq \frac{2}{\gamma\lambda(k+2)^2}\|x^0-\tilde x\|^2.
\end{equation*}
%The optimal choice for $\gamma$ is then again $\gamma_*=\frac{\sqrt{2}-1}{L_f}$,
%for which
%$$F(z^k)-F_\star \leq \frac{L_f}{3-2\sqrt{2}}\frac{4}{(k+2)^2}\|x^0-\tilde x\|^2.$$
\end{theorem}
\begin{proof}
%See \cite{patrinos2014douglas}.
%\begin{comment}
The iterations correspond to the optimal method described in~\cite[Sec. 6.10.2]{bertsekas2009online},
applied to $h$. By \cite[Prop. 6.10.3]{bertsekas2009online} the iterates satisfy
$$ h(w^k)-h(\tilde w)\leq\frac{2L_h}{(k+2)^2}\|w^0-\tilde w\|^2. $$
Switching to the variable $x = Sw$ we get
\begin{align*}
F_\gamma^{\mathrm{DR}}(x^k)-F_\gamma^{\mathrm{DR}}(\tilde x) &\leq \frac{2L_h}{(k+2)^2}\|x^0-\tilde x\|_{D^{-1}}^2\\
&\leq \frac{1}{\lambda_{\min}(D)}\frac{2L_h}{(k+2)^2}\|x^0-\tilde x\|^2\\
&= \frac{\lambda_{\max}(D)}{\lambda_{\min}(D)}\frac{2L_{F_{\gamma}^{\mathrm{DR}}}}{(k+2)^2}\|x^0-\tilde x\|^2\\
&= \frac{1+\gamma L_f}{\gamma(1-\gamma L_f)}\frac{2}{(k+2)^2}\|x^0-\tilde x\|^2.
\end{align*}
Since $z^k = G_\gamma(x^k)$, the result follows by invoking inequality~\eqref{eq:Ineq2} and Theorem~\ref{th:Equiv}.
%\end{comment}
\end{proof}

The optimal choice for $\gamma$ is again $\gamma_\star=(\sqrt{2}-1)/L_f$.
We similarly obtain complexity bounds for the strongly convex case, as described
in the following result.

\begin{theorem}
If $f$ is strongly convex quadratic, $\gamma < 1/L_f$, $\lambda_k$ are given by~\eqref{eq:lambdass} and
\begin{align*}
%\lambda_k &= \lambda=(1-\gamma L_f)/(1+\gamma L_f),\\
\beta_k &= \frac{1-\sqrt{\mu_h/L_h}}{1+\sqrt{\mu_h/L_h}},
\end{align*}
then the sequence of iterates generated by~\eqref{eq:FDRSfirst}-\eqref{eq:FDRSlast} satisfies
\begin{align*}
F(z^k)-F_\star &\leq \frac{L_h}{\lambda_{\min (D)}} \left(1-\sqrt{\frac{\mu_h}{L_h}}\right)^k\|x^0-x_\star\|^2.
\end{align*}
\end{theorem}
\begin{proof}
%See \cite{patrinos2014douglas}.
%\begin{comment}
The proof proceeds similarly to the previous one. The algorithm corresponds
to iterations \cite[Eq. 2.2.9]{nesterov2003introductory} applied to $h$, and
\cite[Th. 2.2.3]{nesterov2003introductory} tells us that
$$ h(w^k)-h(\tilde w) \leq L_h\left(1-\sqrt{\frac{\mu_h}{L_h}}\right)^k\|w^0-\tilde w\|^2.$$
The latter is equivalent to
\begin{align*}
F_\gamma^{\mathrm{DR}}(x^k)-F_\gamma^{\mathrm{DR}}(\tilde x) &\leq L_h\left(1-\sqrt{\frac{\mu_h}{L_h}}\right)^k\|x^0-\tilde x\|_{D^{-1}}^2 \\
    &\leq \frac{L_h}{\lambda_{\min}(D)}\left(1-\sqrt{\frac{\mu_h}{L_h}}\right)^k\|x^0-\tilde x\|^2.
\end{align*}
Again, $z^k = G_\gamma(x^k)$, Theorem~\ref{th:Equiv} and inequality~\eqref{eq:Ineq2}
 complete the result.
%\end{comment}
\end{proof}

%%%%%%%%%%%%%%%%%%%%%%%%%%%%%%%%%%%%%%%%%%%%%%%%%%%%%%%%%%%%%%%%%%%%%%%%%%%%%%%%
\section{Simulations}\label{sec:Simulations}

\subsection{Box-constrained QP}

We tested our analysis against numerical results obtained by
applying the considered methods to the following box-constrained convex quadratic
program
\begin{align*}
\minimize\ &\ \tfrac{1}{2}x'Qx + q'x\\
\stt\ &\ l\leq x \leq u,
\end{align*}
where $Q\in\Re^{n\times n}$ is symmetric and positive semidefinite, while
$q,l,u\in\Re^n$.
The problem is expressed in composite form by setting
$$f(x) = \tfrac{1}{2}x'Qx + q'x,\quad g(x) = \delta_{[l,u]}(x),$$
where $\delta_C$ is the indicator function of the convex set $C$.
As it was pointed out in Section~\ref{sec:Quadratic},
the proximal mapping associated with $f$ is linear
$$ \prox_{\gamma f}(x) = (I+\gamma Q)^{-1}(x-\gamma q). $$
The proximal mapping associated with $g$ is simply the projection
onto the $[l,u]$ box, $\prox_{\gamma g}(x) = \Pi_{[l,u](x)}$.
Tests were performed on problems generated randomly as described
in~\cite{gonzaga2013optimal}. In Figure~\ref{fig:BoxQP_gamma} we illustrate
the performance of DRS for different choices of the parameter $\gamma$. Figure~\ref{fig:BoxQP_accel}
compares the standard DRS and the accelerated method~\eqref{eq:FDRSfirst}-\eqref{eq:FDRSlast}.

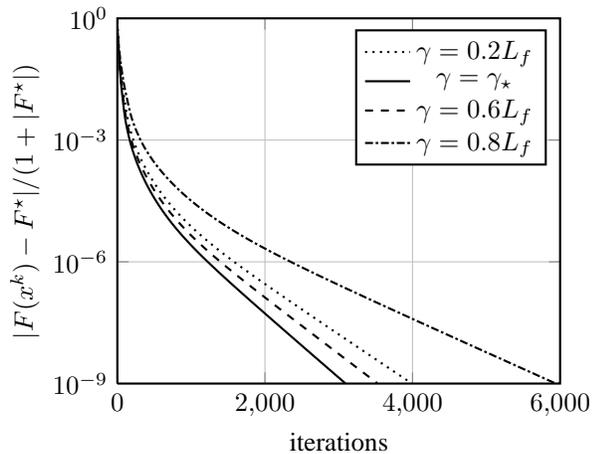
\begin{figure}[tb]
\centering
\begin{tikzpicture}\begin{semilogyaxis}[width=0.42\textwidth,legend pos=north east,xlabel={iterations},ylabel={$\vert F(x^k)-F^\star \vert/(1+\vert F^\star\vert)$},grid=major,thick,xmin=0,xmax=6000,ymin=1e-9,ymax=1e0]
    \addplot[black,dotted] table[x index=0,y index=1,header=false] {figures/QPbox_gamma_0_2.dat};
    \addplot[black] table[x index=0,y index=1,header=false] {figures/QPbox_gamma_star.dat};
    \addplot[black,dashed] table[x index=0,y index=1,header=false] {figures/QPbox_gamma_0_6.dat};
    \addplot[black,dash pattern=on 3pt off 1pt on 1pt off 1pt] table[x index=0,y index=1,header=false] {figures/QPbox_gamma_0_8.dat};
    \legend{$\gamma = 0.2 L_f$,$\gamma = \gamma_\star$,$\gamma = 0.6 L_f$,$\gamma = 0.8 L_f$}
\end{semilogyaxis}\end{tikzpicture}
\caption{DRS applied to a randomly generated box-constrained QP,
with $n=500$, for different values of $\gamma$.}
\label{fig:BoxQP_gamma}
\end{figure}

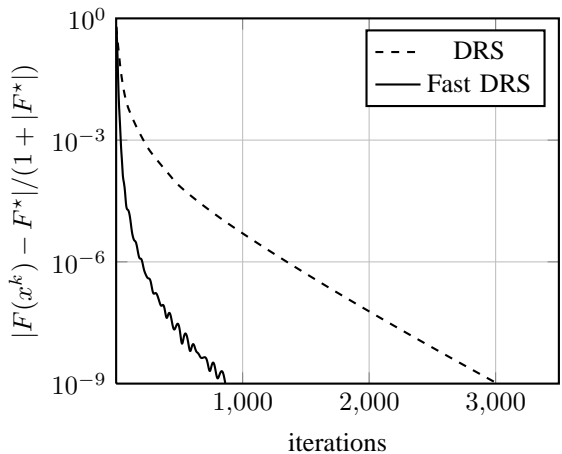
\begin{figure}[tb]
\centering
\begin{tikzpicture}\begin{semilogyaxis}[width=0.42\textwidth,legend pos=north east,xlabel={iterations},ylabel={$\vert F(x^k)-F^\star \vert/(1+\vert F^\star\vert)$},grid=major,thick,xmin=0,xmax=3500,ymin=1e-9,ymax=1e0,xtick={1000,2000,3000}]
    \addplot[black,dashed] table[x index=0,y index=1,header=false] {figures/QPbox_unaccel.dat};
    \addplot[black] table[x index=0,y index=1,header=false] {figures/QPbox_accel.dat};
    \legend{DRS, Fast DRS}
\end{semilogyaxis}\end{tikzpicture}
\caption{Comparison between DRS and its accelerated variant,
for $\gamma = \gamma_\star$, applied to a randomly generated box-constrained QP with $n=500$.}
\label{fig:BoxQP_accel}
\end{figure}

\subsection{Sparse least squares}

The well known $\ell_1$-regularized least squares problem consists of finding 
a sparse solution to an underdetermined linear system. The goal is achieved by
solving
\begin{align*}
\minimize\ &\ \tfrac{1}{2}\|Ax-b\|_2^2 + \rho\|x\|_1,
\end{align*}
where $A\in\Re^{m\times n}$ and $b\in\Re^m$.
The regularization parameter $\rho$ modulates between a low residual
$\|Ax-b\|_2^2$ and a sparse solution.
In this case the proximal mapping with respect to $f$ is
$$ \prox_{\gamma f}(x) = (A'A+\gamma^{-1}I)^{-1}(A'b+\gamma^{-1}x), $$
while $\prox_{\gamma g}$ is the following \emph{soft-thresholding} operator,
$$ \left[\prox_{\gamma g}(x)\right]_i = \sign(x_i)\cdot\max\{0,|x_i|-\gamma\rho\},\ i=1,\ldots n. $$
Random problems were generated according to~\cite{lorenz2013constructing}, and the
results are shown in Figure~\ref{fig:L1LS_gamma} and~\ref{fig:L1LS_accel},
where we compare different choices for $\gamma$ and the fast Douglas-Rachford
iterations.

\begin{figure}[tb]
\centering
\begin{tikzpicture}\begin{semilogyaxis}[width=0.41\textwidth,legend pos=north east,xlabel={iterations},ylabel={$\vert F(x^k)-F^\star \vert/(1+\vert F^\star\vert)$},grid=major,thick,xmin=0,xmax=3000,ymin=1e-9,ymax=1e0]
    \addplot[black,dotted] table[x index=0,y index=1,header=false] {figures/L1LS_gamma_0_2.dat};
    \addplot[black] table[x index=0,y index=1,header=false] {figures/L1LS_gamma_star.dat};
    \addplot[black,dashed] table[x index=0,y index=1,header=false] {figures/L1LS_gamma_0_6.dat};
    \addplot[black,dash pattern=on 3pt off 1pt on 1pt off 1pt] table[x index=0,y index=1,header=false] {figures/L1LS_gamma_0_8.dat};
    \legend{$\gamma = 0.2 L_f$,$\gamma = \gamma_\star$,$\gamma = 0.6 L_f$,$\gamma = 0.8 L_f$}
\end{semilogyaxis}\end{tikzpicture}
\caption{Comparison of different choices of $\gamma$ for a random $\ell_1$ least
squares problem, with $m=100, n=1000$.}
\label{fig:L1LS_gamma}
\end{figure}
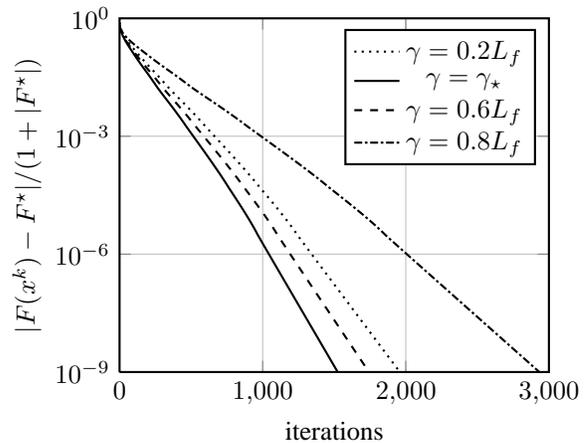

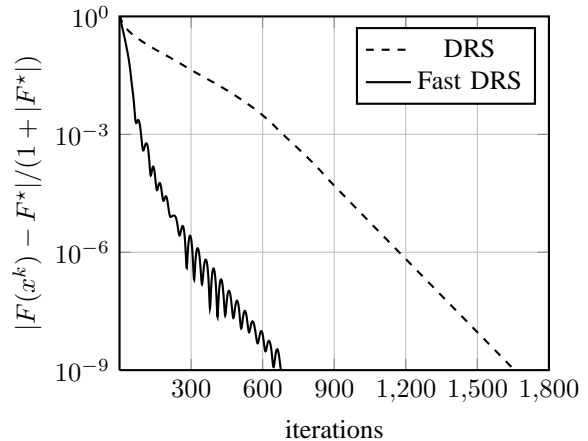
\begin{figure}[tb]
\centering
\begin{tikzpicture}\begin{semilogyaxis}[width=0.41\textwidth,legend pos=north east,xlabel={iterations},ylabel={$\vert F(x^k)-F^\star \vert/(1+\vert F^\star\vert)$},grid=major,thick,xmin=0,xmax=1800,ymin=1e-9,ymax=1e0,xtick={300,600,900,1200,1500,1800}]]
    \addplot[black,dashed] table[x index=0,y index=1,header=false] {figures/L1LS_unaccel.dat};
    \addplot[black] table[x index=0,y index=1,header=false] {figures/L1LS_accel.dat};
    \legend{DRS, Fast DRS}
\end{semilogyaxis}\end{tikzpicture}
\caption{DRS and its accelerated variant, with $\gamma = \gamma_\star$, applied to a
random sparse least squares problem of size $m=100, n=1000$.}
\label{fig:L1LS_accel}
\end{figure}

%\note{General QP (ADMM)}

%%%%%%%%%%%%%%%%%%%%%%%%%%%%%%%%%%%%%%%%%%%%%%%%%%%%%%%%%%%%%%%%%%%%%%%%%%%%%%%%
\section{Conclusions \& Future Work}\label{sec:Conclusions}

%\note{Extend the analysis to the case where $f$ is not quadratic?}

%\note{Extension to the case where $f$ is not twice continuously differentiable?}

In this paper we dealt with convex composite minimization problems.
We introduced a continuously differentiable function, namely
the Douglas-Rachford Envelope (DRE). Its minimizers, under suitable assumptions,
are in a one-to-one correspondence with the solutions of the original convex
composite optimization problem. We observed how the 
DRS iterations, for finding zeros of the sum of two maximal monotone
operators $A$ and $B$, are equivalent to a scaled unconstrained gradient method
applied to the DRE, when $A = \partial f$ and $B =\partial g$ and $f$ is twice
continuously differentiable with Lipschitz continuous gradient.
This allowed us to to apply well-known results of smooth unconstrained optimization
to analyze the convergence of DRS in the particular case of $f$ being convex quadratic.
Moreover, we have been able to apply and
analyze optimal first-order methods and obtain a fast Douglas-Rachford
splitting method.
Ongoing work on this topic include exploiting the illustrated results to study
convergence properties of ADMM.
%Future work includes extending the analysis to the most general case of $f$ being
%convex nonsmooth, exploring the cases where $\gamma>1/L_f$, and using the DRE to 
%devise new Douglas-Rachford splitting variants based on  semismooth Newton methods.

\bibliographystyle{IEEEtran}
\bibliography{DRE_CDC}

%%%%%%%%%%%%%%%%%%%%%%%%%%%%%%%%%%%%%%%%%%%%%%%%%%%%%%%%%%%%%%%%%%%%%%%%%%%%%%%%
\appendix

\section{}

We provide here all the proofs and technical lemmas omitted in the article.

\emph{Proof of Proposition~\ref{prop:BasicInequalities}:}
First we will need the following lemma.
\begin{lemma}\label{le:BasicLemma}
Suppose that $h:\Re^n\to\Rinf$ is proper, closed, convex. Then for all $y\in\Re^n$, $z\in\Re^n$
\begin{align*}
h(z)+\tfrac{1}{2\gamma}\|z-y\|^2 &\geq h(\prox_{\gamma h}(y))+\tfrac{1}{2\gamma}\|\prox_{\gamma h}(y)-y\|^2 \\
    &\phantom{\geq} +\tfrac{1}{2\gamma}\|z-\prox_{\gamma h}(y)\|^2.
\end{align*}
\end{lemma}
\begin{proof}
Let us denote, for brevity, $y_\gamma = \prox_{\gamma h}(y)$.
Function $\phi(z) = \tfrac{1}{2\gamma}\|z-y\|^2$ is strongly convex with modulus
$\gamma^{-1}$. For any $v\in\partial h(y_\gamma)$ we have, by strong convexity of $h(z)+\phi(z)$,
\begin{align*}
h(z)+\phi(z) &= h(z)+\tfrac{1}{2\gamma}\|z-y\|^2 \\
    &\geq h(y_\gamma) + \tfrac{1}{2\gamma}\|y_\gamma-y\|^2 \\
    &\phantom{\geq} + (v + \tfrac{1}{\gamma}(y_\gamma-y))'(z-y_\gamma)\\
    &\phantom{\geq} + \tfrac{1}{2\gamma}\|z-y_\gamma\|^2.
\end{align*}
%    &= h(y_\gamma) + \tfrac{1}{2\gamma}\|y_\gamma-y\|^2+\tfrac{1}{2\gamma}\|z-y_\gamma\|^2,
The result follows by considering $v=\tfrac{1}{\gamma}(y-y_\gamma)$,
which is an element of $\partial h(y_\gamma)$ by the optimality condition for $\prox_{\gamma h}(y)$.
\end{proof}
Now we can proceed with the proof of Proposition~\ref{prop:BasicInequalities}. Due to~\eqref{eq:DRE3}, an alternative expression for the DRE is the following
\begin{align}\label{eq:DRE4}
F^{\mathrm{DR}}_\gamma(x) &= f(P_\gamma(x))+g(G_\gamma(x))+\tfrac{1}{2\gamma}\|G_\gamma(x)-P_\gamma(x)\|^2\nonumber\\
    &\phantom{=} +\gamma^{-1}(G_\gamma(x)-P_\gamma(x))'(x-P_\gamma(x)).
\end{align}
In order to obtain~\eqref{eq:Ineq1}, apply Lemma~\ref{le:BasicLemma}
for $h=g$, $y=2P_\gamma(x)-x$. We have that for all $z\in\Re^n$
\begin{align*}
g(z)&+\tfrac{1}{2\gamma}\|z-(2P_\gamma(x)-x)\|^2 \\
& \geq g(G_\gamma(x))+\tfrac{1}{2\gamma}\|G_\gamma(x)-(2P_\gamma(x)-x)\|^2\\
&+\tfrac{1}{2\gamma}\|z-G_\gamma(x)\|^2.
\end{align*}
Putting $z=P_\gamma(x)$ in the above,
\begin{align*}
g(P_\gamma(x))  & +\tfrac{1}{2\gamma}\|x-P_\gamma(x)\|^2 \\
        & \geq g(G_\gamma(x))+\tfrac{1}{2\gamma}\|G_\gamma(x)-P_\gamma(x)+x-P_\gamma(x)\|^2\\
        & \phantom{\geq} +\tfrac{1}{2\gamma}\|P_\gamma(x)-G_\gamma(x)\|^2\\
        & = g(G_\gamma(x))+\tfrac{1}{2\gamma}\|G_\gamma(x)-P_\gamma(x)\|^2\\
        & \phantom{=} +\tfrac{1}{2\gamma}\|x-P_\gamma(x)\|^2\\
        & \phantom{=} +\gamma^{-1}(G_\gamma(x)-P_\gamma(x))'(x-P_\gamma(x))\\
        & \phantom{=} +\tfrac{1}{2\gamma}\|P_\gamma(x)-G_\gamma(x)\|^2.
\end{align*}
Therefore,
\begin{align*}
g(P_\gamma(x)) & \geq g(G_\gamma(x))+\tfrac{1}{2\gamma}\|G_\gamma(x)-P_\gamma(x)\|^2\\
        & \phantom{\geq} +\gamma^{-1}(G_\gamma(x)-P_\gamma(x))'(x-P_\gamma(x))\\
        & \phantom{\geq} +\tfrac{1}{2\gamma}\|P_\gamma(x)-G_\gamma(x)\|^2.
\end{align*}
Adding $f(P_\gamma(x))$ to both sides,
\begin{align*}
F(P_\gamma(x)) & \geq f(P_\gamma(x))+g(G_\gamma(x))+\tfrac{1}{2\gamma}\|G_\gamma(x)-P_\gamma(x)\|^2\nonumber\\
        & \phantom{\geq} +\gamma^{-1}(G_\gamma(x)-P_\gamma(x))'(x-P_\gamma(x))\\
        & \phantom{\geq} +\tfrac{1}{2\gamma}\|P_\gamma(x)-G_\gamma(x)\|^2.
\end{align*}
We obtain the result by recalling~\eqref{eq:DRE4}.
%For the second inequality,
%we have that, by strong convexity of $f$,
%\begin{align*}
%f(P_\gamma(x)) &\geq f(G_\gamma(x))\\
%    &\phantom{\geq} +\nabla f(P_\gamma(x))'(P_\gamma(x)-G_\gamma(x))\\
%    &\phantom{\geq} +\tfrac{\mu_f}{2}\|P_\gamma(x)-G_\gamma(x)\|^2.
%\end{align*}
%Since $\nabla f(P_\gamma(x))=(1/\gamma)(x-P_\gamma(x))$,
%\begin{align*}
%f(P_\gamma(x)) &\geq f(G_\gamma(x))\\
%    &\phantom{\geq} +(1/\gamma)(x-P_\gamma(x))'(P_\gamma(x)-G_\gamma(x))\\
%    &\phantom{\geq} +\tfrac{\mu_f}{2}\|P_\gamma(x)-G_\gamma(x)\|^2.
%\end{align*}
%Therefore
%\begin{align*}
%F(P_\gamma(x)) &\geq f(P_\gamma(x))+g(G_\gamma(x))\\
%    &\phantom{\geq} +\tfrac{1}{2\gamma}\|G_\gamma(x)-P_\gamma(x)\|^2\\
%    &\phantom{\geq} +\gamma^{-1}(G_\gamma(x)-P_\gamma(x))'(x-P_\gamma(x))\\
%    &\phantom{\geq} +\tfrac{1+\gamma\mu_g}{2\gamma}\|P_\gamma(x)-G_\gamma(x)\|^2\\
%    &\geq f(G_\gamma(x))+\tfrac{\mu_f}{2}\|P_\gamma(x)-G_\gamma(x)\|^2\\
%    &\phantom{\geq} +g(G_\gamma(x))+\tfrac{1}{2\gamma}\|G_\gamma(x)-P_\gamma(x)\|^2\\
%    &\phantom{\geq} +\tfrac{1+\gamma\mu_g}{2\gamma}\|P_\gamma(x)-G_\gamma(x)\|^2\\
%    & = F(G_\gamma(x))+\tfrac{2+\gamma(\mu_f+\mu_g)}{2\gamma}\|P_\gamma(x)-G_\gamma(x)\|^2.
%\end{align*}
Inequality~\eqref{eq:Ineq2} is obtained as follows,
\begin{align*}
F(G_\gamma(x))&= f(G_\gamma(x))+g(G_\gamma(x))\\
    &\leq f(P_\gamma(x))+g(G_\gamma(x))\\
    &\phantom{\leq}+\nabla f(P_\gamma(x))'(G_\gamma(x)-P_\gamma(x))\\
    &\phantom{\leq}+\tfrac{L_f}{2}\|G_\gamma(x)-P_\gamma(x)\|^2\\
    &=f(P_\gamma(x))+g(G_\gamma(x))\\
    &\phantom{=} +\gamma^{-1}(G_\gamma(x)-P_\gamma(x))'(x-P_\gamma(x))\\
    &\phantom{=} +\tfrac{L_f}{2}\|G_\gamma(x)-P_\gamma(x)\|^2\\
    &=F^{\mathrm{DR}}_\gamma(x)-\tfrac{1-\gamma L_f}{2\gamma}\|G_\gamma(x)-P_\gamma(x)\|^2,
\end{align*}
where the first inequality follows from the Lipschitz continuity of $\nabla f$
and the last equality from~\eqref{eq:DRE4}.
\hfill$\blacksquare$

The next basic result is used in the proof of Theorem~\ref{th:QuadConvex}.
\begin{lemma}\label{le:NonexpZ}
Mapping $Z_\gamma:\Re^n\to\Re^n$ is nonexpansive.
\end{lemma}
\begin{proof}
We can express $Z_\gamma$ as
$$ Z_\gamma(x) = \tfrac{1}{2}(x-T(x)),$$
where $T=R_{\gamma\partial g}\circ R_{\gamma\partial f}$ and $R_{\gamma\partial f}$,
$R_{\gamma\partial g}$ are called \emph{reflected resolvent}~\cite[Chap. 23]{bauschke2011convex}
of $\partial f$ and $\partial g$, respectively. Reflected resolvents of maximal
monotone mappings (such as the subdifferential of a convex function) are known
to be nonexpansive~\cite[Cor. 23.10]{bauschke2011convex}, and so is their composition $T$.
Then we have
$$\|T(x_1)-T(x_2)\|\leq \|x_1-x_2\|,$$
for all $x_1,x_2\in\Re^n$, or
$$\|-2(Z_\gamma(x_1)-Z_\gamma(x_2))+(x_1-x_2)\|\leq\|x_1-x_2\|.$$
Using the reverse triangle inequality
$$2\|Z_\gamma(x_1)-Z_\gamma(x_2)\|-\|x_1-x_2\|\leq\|x_1-x_2\|,$$
or
$$\|Z_\gamma(x_1)-Z_\gamma(x_2)\|\leq \|x_1-x_2\|,$$
\ie $Z_\gamma$ is nonexpansive.
\end{proof}

\end{document}